\documentclass[12pt]{article}
\usepackage{amsfonts}
\usepackage{amsmath, amssymb, amsthm}
\usepackage{enumerate}
\usepackage{epsfig}
\usepackage{verbatim}
\usepackage{stmaryrd}
\usepackage{graphicx}
\usepackage{hyperref}

\newtheorem{theorem}{Theorem}

\newtheorem{assumption}{Assumption}

\newtheorem{definition}[theorem]{Definition}
\newtheorem{lemma}[theorem]{Lemma}

\input{tcilatex}
\begin{document}

\title{Mixed $hp$ FEM for singularly perturbed fourth order boundary value
problems with two small parameters}
\author{C. Xenophontos$^{\ddagger ,}$\thanks{%
Corresponding Author. Email: xenophontos@ucy.ac.cy}, \hspace{0.01cm} S. Franz%
\thanks{%
Technische Universit\"{a}t Dresden, Institut f\"{u}r Wissenschaftliches
Rechnen, 01062 Dresden, Germany} \hspace{0.01cm} and I. Sykopetritou\thanks{%
Department of Mathematics and Statistics, University of Cyprus, P.O. Box
20537, 1678 Nicosia, Cyprus}. }
\maketitle

\begin{abstract}
We consider fourth order singularly perturbed boundary value problems with
two small parameters, and the approximation of their solution by the $hp$
version of the Finite Element Method on the {\emph{Spectral Boundary Layer}}
mesh from \cite{MXO}. We use a mixed formulation requiring only $C^{0}$
basis functions in two-dimensional smooth domains. Under the assumption of
analytic data, we show that the method converges uniformly, with respect to
both singular perturbation parameters, at an exponential rate when the error
is measured in the energy norm. Our theoretical findings are illustrated
through numerical examples, including results using a stronger (balanced)
norm.
\end{abstract}


\bigskip


\section{Introduction}

\label{intro}

Fourth order singularly perturbed problems (SPPs) have not been studied as
much as their second order counter-parts for which there is considerable
literature (see, \cite{mos}, \cite{morton}, \cite{rst} and the references
therein). As is well known, SPPs give rise to \emph{boundary layers} in the
solution (for second order problems) and in its derivative (for fourth order
problems) \cite{OMalley}. If the problem contains two parameters, then
second order problems can become reaction-convection-diffusion, with two
layers of different scales at different parts of the boundary \cite{L}.
Fourth order SPPs with two parameters tend to give rise to boundary layers
in both the solution and its derivative (of different scales for each) \cite%
{OMalley}. The numerical scheme designed for the robust approximation of the
solution should take these phenomena into account and, in the context of
Finite Differences or Finite Elements, \emph{layer adapted} meshes should be
used (see \cite{L} and the references therein).

In this article, we consider a fourth order reaction-diffusion type SPP that
includes two different parameters $\varepsilon _{1}^{2}$ and $\varepsilon
_{2}^{2}$, multiplying the fourth and second derivatives, respectively. The
relationship between $\varepsilon _{1}$ and $\varepsilon _{2}$ determines
the type of problem we have; we focus on the case when $\varepsilon
_{1}<\varepsilon _{2}$ (see eq. (\ref{relation}) ahead) and we expect a
boundary layer of width $O\left( \varepsilon _{2}\right) $ in $u$ and one of
width $O\left( \frac{\varepsilon _{1}}{\varepsilon _{2}}\right)$ in the
derivative of $u$. We consider two-dimensional smooth domains (with the
boundary given by an analytic curve) and cast the problem in a \emph{mixed}
variational formulation. This allows for a $C^{0}$ discretization, which is
based on the \emph{Spectral Boundary Layer} mesh from \cite{MXO}. Under the
analytic regularity of the data assumption, we show that the method
converges uniformly and exponentially fast, when the error is measured in
the energy norm. Finally, we comment on the use of a stronger `balanced'
norm, as an error measure. All theoretical findings are illustrated by
numerical computations.

The rest of the paper is organized as follows: in Section \ref{2d} we
present the model problem and the assumptions we make regarding the
regularity of its solution. The variational formulation, the discretization
using the \emph{Spectral Boundary Layer} mesh and the proof of uniform,
exponential convergence are presented in Section \ref{hpFEM}. Finally,
Section \ref{nr} shows the results of numerical computations that illustrate
the theoretical findings.

With $\Omega \subset \mathbb{R}^{2},\ $a domain with boundary $\partial
\Omega $ and measure $\left\vert \Omega \right\vert $, we will denote by $%
C^{k}(\Omega )$ the space of continuous functions on $\Omega $ with
continuous derivatives up to order $k$. We will use the usual Sobolev spaces 
$W^{k,m}(\Omega )$ of functions on $\Omega $ with $0,1,2,...,k$ generalized
derivatives in $L^{m}\left( \Omega \right) $, equipped with the norm and
seminorm $\left\Vert \cdot \right\Vert _{k,m,\Omega }$ and $\left\vert \cdot
\right\vert _{k,m,\Omega }\,$, respectively. When $m=2$, we will write $%
H^{k}\left( \Omega \right) $ instead of $W^{k,2}\left( \Omega \right) $, and
for the norm and seminorm, we will write $\left\Vert \cdot \right\Vert
_{k,\Omega }$ and $\left\vert \cdot \right\vert _{k,\Omega }\,$,
respectively. The usual $L^{2}(\Omega )$ inner product will be denoted by $%
\left\langle \cdot ,\cdot \right\rangle _{\Omega }$, with the subscript
ommitted when there is no confusion. We will also use the space 
\begin{equation*}
H_{0}^{1}\left( \Omega \right) =\left\{ u\in H^{1}\left( \Omega \right)
:\left. u\right\vert _{\partial \Omega }=0\right\} .
\end{equation*}%
The norm of the space $L^{\infty }(\Omega )$ of essentially bounded
functions is denoted by $\Vert \cdot \Vert _{\infty ,\Omega }$. Finally, the
notation \textquotedblleft $a\lesssim b$\textquotedblright\ means
\textquotedblleft $a\leq Cb$\textquotedblright\ with $C$ being a generic
positive constant, independent of any discretization or singular
perturbation parameters.


\section{The model problem and its regularity\label{2d}}

\label{step2D}The regularity of the solution to any partial differential
equation (PDE) is governed by the data of the problem as well as the domain $%
\Omega $, where it is posed. If $\Omega $ contains corners, e.g. it is
polygonal, then the solution of singularly perturbed elliptic PDEs will
contain corner singularities, in addition to boundary layers (see \cite%
{melenk} and the references therein). The interplay of the two is a
complicated affair and has only been studied in a few special cases \cite%
{HanKellogg}, \cite{KelloggStynes}. If $\Omega $ is a smooth domain, then
only boundary layers will be present, but the discretization must include
(at least some) curved elements. Since $C^{1}$ curved elements are difficult
to construct, we use a mixed formulation which only requires a $C^{0}$
discretization.

We consider the following model problem: find $u$ such that 
\begin{eqnarray}
\varepsilon _{1}^{2}\Delta ^{2}u-\varepsilon _{2}^{2}\Delta u+c(x,y)u
&=&f(x,y)\text{ in }\Omega \subset \mathbb{R}^{2},  \label{pde} \\
u=\frac{\partial u}{\partial n} &=&0\text{ on }\partial \Omega .
\label{pde_bc}
\end{eqnarray}%
where $0<\varepsilon _{1},\varepsilon _{2}\leq 1$ are given parameters that
can approach zero and the functions $c,f$ are given and sufficiently smooth.
\ In particular, we assume that they are analytic functions satisfying, for
some positive constants $\gamma _{f},\gamma _{c}$, independent of $%
\varepsilon _{1},\varepsilon _{2},$%
\begin{equation}
\left\Vert \nabla ^{n}f\right\Vert _{\infty ,\Omega }\lesssim n!\gamma
_{f}^{n},\left\Vert \nabla ^{n}c\right\Vert _{\infty ,\Omega }\lesssim
n!\gamma _{c}^{n},\;\forall \;n\in \mathbb{N}_{0}.  \label{analytic2}
\end{equation}%
Here we have used the shorthand notation 
\begin{equation*}
\left\vert \nabla ^{n}f\right\vert ^{2}:=\sum_{\left\vert \alpha \right\vert
=n}\frac{\left\vert \alpha \right\vert !}{\alpha !}\left\vert D^{\alpha
}f\right\vert ^{2},
\end{equation*}%
with $D^{m}$ denoting differentiation of order $|m|$. In addition, we assume
that there exists a constant $\gamma $, independent of $\varepsilon
_{1},\varepsilon _{2},$ such that $\forall \;(x,y)\in \bar{\Omega}$ 
\begin{equation}
c(x,y)\geq \gamma >0,  \label{data2}
\end{equation}%
and that $\Omega $ is a smooth domain meaning that $\partial \Omega $ is an
analytic curve. We mention that problem (\ref{pde}), (\ref{pde_bc}) has been
studied in \cite{CFLX} (see also \cite{PC}) with $\varepsilon _{2}$ a fixed
constant, hence only one singular perturbation parameter. Here we will focus
on the case $\varepsilon _{1}<\varepsilon _{2},$ and in particular we assume%
\begin{equation}
\varepsilon _{1}\lesssim \varepsilon _{2}^{2}<<1,  \label{relation}
\end{equation}%
since in the complementary case we have a typical (`one parameter')
reaction-diffusion SPP.

We introduce a new unknown{\footnote{%
The fact that $w\in H^{2}(\Omega )$ is due to $\Omega $ being a smooth
domain.}} $w=\varepsilon _{1}\Delta u\in H^{2}(\Omega )$ and cast the
problem in the following \emph{mixed formulation}: find $\left( u,w\right) $
such that 
\begin{equation}
\left. 
\begin{array}{c}
\varepsilon _{1}\Delta u-w=0\text{ in }\Omega \\ 
\varepsilon _{1}\Delta w-\varepsilon _{2}^{2}\Delta u+cu=f\text{ in }\Omega
\\ 
u=\frac{\partial u}{\partial n}=0\text{ on }\partial \Omega%
\end{array}%
\right\} .  \label{mixed}
\end{equation}%
The variational formulation of (\ref{mixed}) reads: find $u\in
H_{0}^{1}(\Omega ),w\in H^{1}(\Omega )$ such that%
\begin{equation}
\mathcal{A}\left( (u,w),(\psi ,\varphi )\right) =\left\langle f,\psi
\right\rangle _{\Omega }\;\forall \;(\psi ,\varphi )\in H_{0}^{1}(\Omega
)\times H^{1}(\Omega ),  \label{AuwFw}
\end{equation}%
where%
\begin{equation}
\mathcal{A}\left( (u,w),(\psi ,\varphi )\right) =\left\langle cu,\psi
\right\rangle _{\Omega }+\varepsilon _{2}^{2}\left\langle \nabla u,\nabla
\psi \right\rangle _{\Omega }+\left\langle w,\phi \right\rangle _{\Omega
}+\varepsilon _{1}\left\langle \nabla u,\nabla \phi \right\rangle _{\Omega
}-\varepsilon _{1}\left\langle \nabla w,\nabla \psi \right\rangle _{\Omega }.
\label{Auw}
\end{equation}%
The bilinear form (\ref{Auw}) is \emph{coercive}, i.e.%
\begin{equation}
\mathcal{A}\left( (u,w),(u,w)\right) \gtrsim |||(u,w)|||^{2},
\label{coercive}
\end{equation}%
where the \emph{energy} norm is defined as%
\begin{equation}
|||(u,w)|||^{2}=\left\Vert u\right\Vert _{0}^{2}+\varepsilon
_{2}^{2}\left\Vert \nabla u\right\Vert _{0}^{2}+\left\Vert w\right\Vert
_{0}^{2},  \label{norm}
\end{equation}%
and is equivalent to the classical weighted $H^{2}$ norm:%
\begin{equation*}
|||(u,w)|||^{2}=|||(u,\varepsilon _{1}\Delta u)|||^{2}=\left\Vert
u\right\Vert _{0}^{2}+\varepsilon _{2}^{2}\left\Vert \nabla u\right\Vert
_{0}^{2}+\varepsilon _{1}^{2}\left\Vert \Delta u\right\Vert _{0}^{2}.
\end{equation*}

Since the domain is assumed to be smooth, \emph{boundary fitted} coordinates 
$(\rho ,\theta )$ are appropriate (see, e.g. \cite{AF}): Let $\left(
X(\theta ),Y(\theta )\right) ,\theta \in \lbrack 0,L]$ be a parametrization
of $\partial \Omega $ by arclength and let $\Omega _{0}$ be a tubular
neighborhood of $\partial \Omega $ in $\Omega $. For each point $z=(x,y)\in
\Omega _{0}$ there is a unique nearest point $z_{0}\in \partial \Omega $, so
with $\theta $ the arclength parameter (with counterclockwise orientation),
we set $\rho =\left\vert z-z_{0}\right\vert $ which measures the distance
from the point $z$ to $\partial \Omega .$ Explicitely, 
\begin{equation}
\Omega _{0}=\left\{ z-\rho \overrightarrow{n}_{z}:z\in \partial \Omega
,0<\rho <\rho _{0}<{\text{min. radius of curvature of }}\partial \Omega
\right\} ,  \label{Omega0}
\end{equation}%
where $\overrightarrow{n}_{z}$ is the outward unit normal at $z\in \partial
\Omega $, and 
\begin{equation*}
x=X(\theta )-\rho Y^{\prime }(\theta ),y=Y(\theta )+\rho X^{\prime }(\theta
),
\end{equation*}%
with $\rho \in (0,\rho _{0}),\theta \in (0,L).$ The determinant of the
Jacobian matrix of the transformation is given by $J=1-\kappa (\theta )\rho $%
, where $\kappa (\theta )$ is the curvature of $\partial \Omega $ (see \cite%
{AF} for more details). Since we assume that $\partial \Omega $ is an
analytic curve, we have $X^{(k)}(\theta ),Y^{(k)}(\theta )\lesssim
1\;\forall \;k=0,1,2,...$, as well as $J,J^{-1}\lesssim 1.$ Thus for a
function $v(x,y)$ defined in $\Omega _{0}$, the above change of variables
produces 
\begin{equation*}
v(x,y)=v\left( X(\theta )-\rho Y^{\prime }(\theta ),Y(\theta )+\rho
X^{\prime }(\theta )\right) ,
\end{equation*}
as well as%
\begin{equation*}
\frac{\partial v}{\partial x}=\frac{1}{1-\kappa (\theta )\rho }\left\{ \frac{%
\partial v}{\partial \theta }X^{\prime }(\theta )-\frac{\partial v}{\partial
\rho }\left( Y^{\prime }(\theta )+\rho X^{\prime \prime }(\theta )\right)
\right\} ,
\end{equation*}%
\begin{equation*}
\frac{\partial v}{\partial y}=\frac{1}{1-\kappa (\theta )\rho }\left\{ \frac{%
\partial v}{\partial \rho }\left( X^{\prime }(\theta )-\rho Y^{\prime \prime
}(\theta )\right) +\frac{\partial v}{\partial \theta }Y^{\prime }(\theta
)\right\} .
\end{equation*}%
This shows that the first derivatives with respect to the (physical) $x,y$
variables are bounded by the first derivatives with respect to the $\rho
,\theta $ variables.

In \cite{OMalley} it was shown that the solution to two parameter singularly
perturbed fourth order problems may be decomposed into a smooth part,
boundary layers and a remainder. Derivative estimates for each term in the
decomposition are also given, up to a fixed, low order. The analytic
regularity of the solution is beyond the scope of this article and will
appear elsewhere. Here we make the following assumption, which is in line
with \cite{OMalley} and \cite{BFR}.

\begin{assumption}
\label{assumption2D} Let $(u,w)$ be the solution to (\ref{AuwFw}), and
assume (\ref{analytic2}) holds. Then, there exists a positive constant $K$
such that for all $(x,y)\in \bar{\Omega}$,%
\begin{equation}
\left\vert \frac{\partial ^{n+m}u}{\partial x^{n}\partial y^{m}}\right\vert
\lesssim K^{n+m}\max \left\{ (n+m)^{n+m},\left( \frac{\varepsilon _{1}}{%
\varepsilon _{2}}\right) ^{1-n-m}\right\} ,\;  \label{B1}
\end{equation}%
Moreover, $u$ and $w$ may be decomposed as%
\begin{equation}
u=S+E_{1}+E_{2}+R,w=\tilde{S}+\tilde{E}_{1}+\tilde{E}_{2}+\tilde{R},
\label{B2}
\end{equation}%
and, there exist constants $K_{0},\tilde{K}_{0},K_{1},\tilde{K}_{1},K_{2},%
\tilde{K}_{2},\delta >0$ such that for all $(x,y)\in \bar{\Omega},$ 
\begin{equation}
\left\vert \frac{\partial ^{n+m}S}{\partial x^{n}\partial y^{m}}\right\vert
\lesssim K_{0}^{n+m}(n+m)^{n+m},\left\vert \frac{\partial ^{n+m}\tilde{S}}{%
\partial x^{n}\partial y^{m}}\right\vert \lesssim \tilde{K}%
_{0}^{n+m}(n+m)^{n+m},  \label{B3}
\end{equation}%
\begin{equation}
\left\vert \frac{\partial ^{n+m}E_{1}(\rho ,\theta )}{\partial \rho
^{n}\partial \theta ^{m}}\right\vert \lesssim K_{1}^{n+m}\frac{1}{%
\varepsilon _{2}}\left( \frac{\varepsilon _{1}}{\varepsilon _{2}}\right)
^{1-n}e^{-\varepsilon _{2}\rho /\varepsilon _{1}}\;,\;\left\vert \frac{%
\partial ^{n+m}E_{2}(\rho ,\theta )}{\partial \rho ^{n}\partial \theta ^{m}}%
\right\vert \lesssim K_{2}^{n+m}\varepsilon _{2}^{-n}e^{-\rho /\varepsilon
_{2}},  \label{B4}
\end{equation}%
\begin{equation}
\left\vert \frac{\partial ^{n+m}\tilde{E}_{1}(\rho ,\theta )}{\partial \rho
^{n}\partial \theta ^{m}}\right\vert \lesssim \tilde{K}_{1}^{n+m}\left( 
\frac{\varepsilon _{1}}{\varepsilon _{2}}\right) ^{-n}e^{-\varepsilon
_{2}\rho /\varepsilon _{1}}\;,\;\left\vert \frac{\partial ^{n+m}\tilde{E}%
_{2}(\rho ,\theta )}{\partial \rho ^{n}\partial \theta ^{m}}\right\vert
\lesssim \tilde{K}_{2}^{n+m}\varepsilon _{2}^{-n}e^{-\rho /\varepsilon _{2}},
\label{B5}
\end{equation}%
\begin{equation}
|||(R,\tilde{R})|||\lesssim e^{-\delta \varepsilon _{1}/\varepsilon
_{2}}\;,\;\delta \in \mathbb{R}^{+},  \label{B6}
\end{equation}%
In (\ref{B2}), $S,\tilde{S}$ correspond to the smooth parts, $E_{1},E_{2}$
correspond to the boundary layers in $u$, $\tilde{E}_{2},\tilde{E}_{3}$
correspond to the boundary layers in $w$, and $R,\tilde{R}$ are the
remainders.
\end{assumption}

Equations (\ref{B4}), (\ref{B5}) are identical to those found in \cite%
{OMalley}; the difference lies in (\ref{B1}) and (\ref{B3}) in that we
assume the smooth parts behave like analytic functions (since we assumed the
data is analytic) as opposed to simply having bounded derivatives in an
unspecified way. The proof of (\ref{B1}), (\ref{B3}) is open.

\section{Discretization by a mixed $hp$-FEM\label{hpFEM}}

The discrete version of (\ref{AuwFw}) reads: find $u_{N}\in V_{1}^{N}\subset
H_{0}^{1}(\Omega ),w_{N}\in V_{2}^{N}\subset H^{1}(\Omega )$ such that 
\begin{equation}
\mathcal{A}\left( (u_{N},w_{N}),(\psi ,\varphi )\right) =\left\langle f,\psi
\right\rangle _{\Omega }\;\forall \;(\psi ,\varphi )\in V_{1}^{N}\times
V_{2}^{N}  \label{discrete2}
\end{equation}%
where $V_{1}^{N},V_{2}^{N}$ are finite dimensional spaces to be defined
shortly. Galerkin orthogonality holds:%
\begin{equation*}
\mathcal{A}\left( (u-u_{N},w-w_{N}),(\psi ,\varphi )\right) =0\;\forall
\;(\psi ,\varphi )\in V_{1}^{N}\times V_{2}^{N}.
\end{equation*}%
In order to define the spaces $V_{1}^{N},V_{2}^{N}$, we let $\Delta =\left\{
\Omega _{i}\right\} _{i=1}^{N}$ be a mesh consinsting of curvilinear
quadrilaterals, subject to the usual conditions (see, e.g. \cite{MS}) and
associate with each $\Omega _{i}$ a bijective mapping $M_{i}:\hat{\Omega}%
\rightarrow \overline{\Omega }_{i}$, where $\hat{\Omega}=[0,1]^{2}$ denotes
the reference square. With $Q_{p}(\hat{\Omega})$ the space of polynomials of
degree $p$ (in each variable) on $\hat{\Omega}$, we define 
\begin{eqnarray*}
\mathcal{S}^{p}(\Delta ) &=&\left\{ u\in H^{1}\left( \Omega \right) :\left.
u\right\vert _{\Omega _{i}}\circ M_{i}\in Q_{p}(\hat{\Omega}),\quad
i=1,...,N\right\} , \\
\mathcal{S}_{0}^{p}(\Delta ) &=&\mathcal{S}^{p}(\Delta )\cap
H_{0}^{1}(\Omega ).
\end{eqnarray*}%
We then take $V_{1}^{N}=\mathcal{S}_{0}^{p}(\Delta ),V_{2}^{N}=\mathcal{S}%
^{p}(\Delta )$, with the mesh $\Delta $ chosen following the construction in 
\cite{MS,MXO}: we begin with a \emph{fixed} (asymptotic) mesh $\Delta _{A}$,
consisting of curvilinear quadrilateral elements $\Omega _{i}$, $i=1,\ldots
,N_{1}$, which are the images of the reference square $\hat{\Omega}$ under
the element mappings $M_{A,i}$, $i=1,\ldots ,N_{1}\in \mathbb{N}$ (the
subscript $A$ stands for asymptotic). They are assumed to satisfy conditions
(M1)--(M3) of \cite{MS} in order for the space ${\mathcal{S}}^{p}(\Delta )$
to have the necessary approximation properties. Moreover, the element
mapings $M_{A,i}$ are assumed to be analytic (with analytic inverse). We
also assume that the elements do not have a single vertex on the boundary $%
\partial \Omega $ but only complete, single edges. For convenience, we
number the elements along the boundary first, i.e., $\Omega _{i}$, $%
i=1,\ldots ,N_{2}<N_{1}$ for some $N_{2}\in \mathbb{N}$. We next give the
definition of the two-dimensional \emph{Spectral Boundary Layer Mesh} $%
\Delta _{BL}=\Delta _{BL}(\kappa ,p)$.

\begin{definition}[Spectral Boundary Layer mesh $\Delta _{BL}(\protect\kappa %
,p)$]
\label{SBL-2D}\cite{MXO} Given parameters $\kappa >0$, $p\in \mathbb{N}$, $%
\varepsilon \in (0,1]$ and the (asymptotic) mesh $\Delta _{A}$, the Spectral
Boundary Layer mesh $\Delta _{BL}(\kappa ,p)$ is defined as follows:

\begin{enumerate}
\item If $\kappa p\frac{\varepsilon _{1}}{\varepsilon _{2}}\geq 1/2$ then we
are in the asymptotic range of p and we use the mesh $\Delta _{A}$.

\item If $\kappa p\frac{\varepsilon _{1}}{\varepsilon _{2}}<1/2$, we need to
define so-called boundary layer elements. We do so by splitting the elements 
$\Omega _{i},i=1,\ldots ,N_{2}$ into three elements $\Omega
_{i}^{BL_{1}},\Omega _{i}^{BL_{2}}$ and $\Omega _{i}^{reg}.$ To this end,
split the reference square $\hat{\Omega}$ into three elements 
\begin{equation*}
\hat{\Omega}^{BL_{1}}=\left[ 0,\kappa p\frac{\varepsilon _{1}}{\varepsilon
_{2}}\right] \times \lbrack 0,1]\;,\;\hat{\Omega}^{BL_{2}}=\left[ \kappa p%
\frac{\varepsilon _{1}}{\varepsilon _{2}},\kappa p\varepsilon _{2}\right]
\times \lbrack 0,1]\;,\;\hat{\Omega}^{reg}=\left[ \kappa p\varepsilon _{2},1%
\right] \times \lbrack 0,1],
\end{equation*}%
and define the elements $\Omega _{i}^{BL_{1}}$, $\Omega _{i}^{BL_{1}}$, $%
\Omega _{i}^{reg}$ as the images of these three elements under the element
map $M_{A,i}$ and the corresponding element maps as the concatination of the
affine maps 
\begin{align*}
& A^{BL_{1}}:\hat{\Omega}\rightarrow \hat{\Omega}^{BL_{1}},\qquad (\xi ,\eta
)\rightarrow \left( \kappa p\frac{\varepsilon _{1}}{\varepsilon _{2}}\xi
,\eta \right) , \\
A^{BL_{2}}& :\hat{\Omega}\rightarrow \hat{\Omega}^{BL_{2}},\qquad (\xi ,\eta
)\rightarrow \left( \kappa p(\frac{\varepsilon _{1}}{\varepsilon _{2}}%
-\varepsilon _{2})\xi ,\eta \right) , \\
& A^{reg}:\hat{\Omega}\rightarrow \hat{\Omega}^{reg},\qquad (\xi ,\eta
)\rightarrow (\kappa p\varepsilon _{2}+(1-\kappa p\varepsilon _{2})\xi ,\eta
)
\end{align*}%
with the element map $M_{A,i}$, i.e., $M_{i}^{BL_{k}}=M_{A,i}\circ
A^{BL_{k}},k=1,2$ and $M_{i}^{reg}=M_{A,i}\circ A^{reg}$.
\end{enumerate}
\end{definition}

In total, the mesh $\Delta _{BL}(\kappa ,p)$ consists of $N=N_{1}+2N_{2}$
elements if $\kappa p\frac{\varepsilon _{1}}{\varepsilon _{2}}<1/2$. By
construction, the resulting mesh 
\begin{equation*}
\Delta _{BL}=\Delta _{BL}(\kappa ,p)=\left\{ \Omega _{1}^{BL_{1}},...,\Omega
_{N_{2}}^{BL_{1}},\Omega _{1}^{BL_{2}},...,\Omega _{N_{2}}^{BL_{2}},\Omega
_{1}^{reg},...,\Omega _{N_{1}}^{reg},\Omega _{N_{1}+1},...,\Omega
_{N}\right\}
\end{equation*}%
is a regular admissible mesh in the sense of \cite{MS}.

In Figure 1 we show an example of such a mesh construction on the unit
circle.

\vspace{1cm}

\begin{figure}[th]
\label{SBLM0}
\par
\begin{center}
\includegraphics[width=0.4\textwidth]{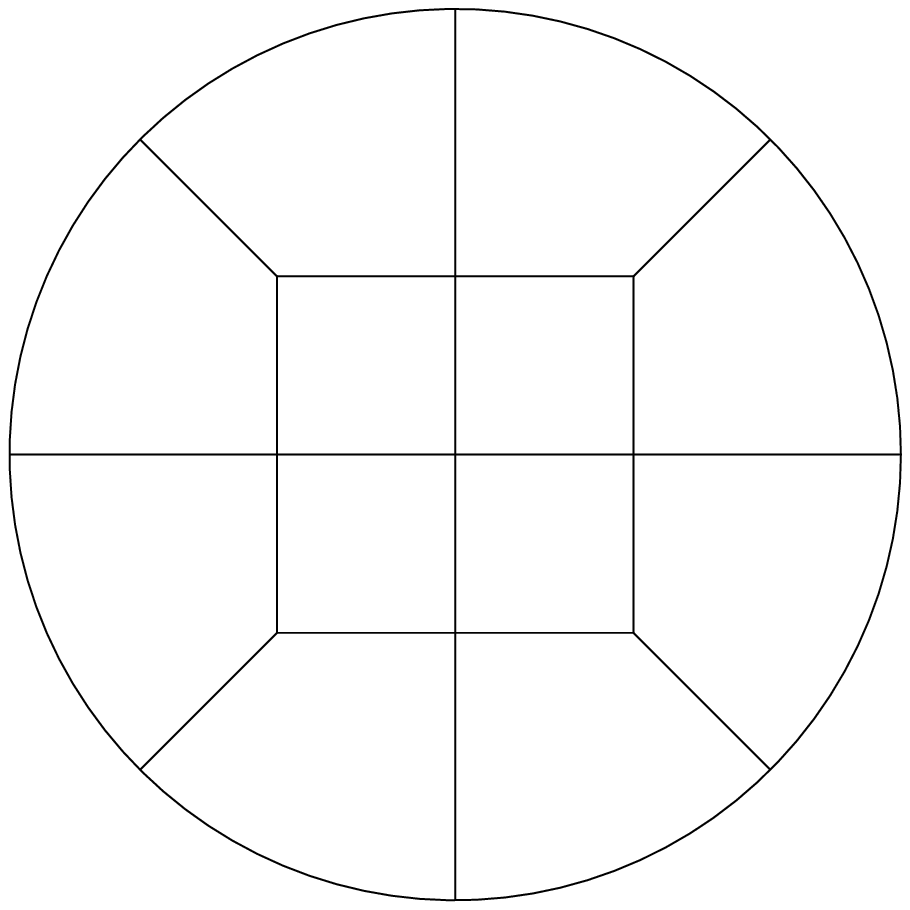} \hspace{2cm} %
\includegraphics[width=0.4\textwidth]{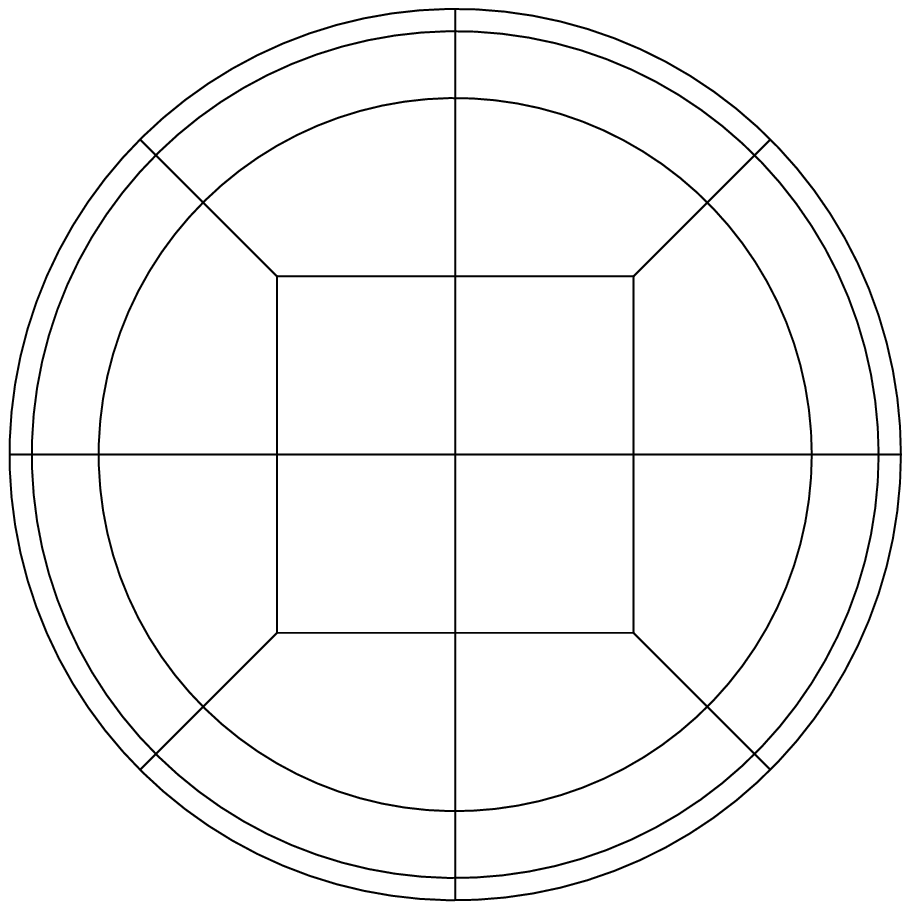}
\end{center}
\par
\begin{picture}(10,10) 
\put(100,220){$\Delta_A$}
\put(340,220){$\Delta_{BL}$}
\put(125,180){$\Omega_1$}
\put(360,170){$\Omega^{reg}_1$}
\put(390,195){$\Omega^{BL_1}_1$}
\put(360,190){$\Omega^{BL_2}_1$}
\put(160,140){$\Omega_2$}
\put(395,130){$\Omega^{reg}_2$}
\put(427,155){$\Omega^{BL_1}_2$}
\put(400,160){$\Omega^{BL_2}_2$}
\put(160,100){$\Omega_3$}
\put(393,100){$\Omega^{reg}_3$}
\put(435,95){$\Omega^{BL_1}_3$}
\put(413,80){$\Omega^{BL_2}_3$}
\put(70,180){$\Omega_{N_1}$}
\put(317,170){$\Omega^{reg}_{N_1}$}
\put(316,205){$\Omega^{BL_1}_{N_1}$}
\put(310,185){$\Omega^{BL_2}_{N_1}$}
\end{picture}
\caption{Example of an admissible mesh. Left: asymptotic mesh $\Delta _{A}$.
Right: boundary layer mesh $\Delta _{BL}$.}
\end{figure}

\subsection{Error Estimates\label{estimates2}}

Our approximation is based on the (element-wise) Gau\ss -Lobatto interpolant
from \cite[Prop. 3.11]{MS}. We have the following.

\begin{lemma}
\label{2Dapprox} Let $\left( u,w\right) $ be the solution to (\ref{mixed})
and assume that (\ref{analytic2}) holds. Then there exist constants $\kappa
_{0}$, $\kappa _{1}$, $C$, $\tilde{\beta}>0$ independent of $\varepsilon \in
(0,1]$ and $p\in \mathbb{N}$, such that the following is true: For every $p$
and every $\kappa \in (0,\kappa _{0}]$ with $\kappa p\geq \kappa _{1},$
there exist $\pi _{p}u\in \mathcal{S}_{0}^{{p}}(\Delta _{BL}(\kappa ,p)),\pi
_{p}w\in \mathcal{S}^{{p}}(\Delta _{BL}(\kappa ,p))$ such that%
\begin{equation*}
\max \left\{ \left\Vert u-\pi _{p}u\right\Vert _{\infty ,\Omega },\left\Vert
w-\pi _{p}w\right\Vert _{\infty ,\Omega }\right\} \lesssim e^{-\tilde{\beta}%
p},
\end{equation*}%
\begin{equation*}
\max \left\{ \varepsilon _{2}^{1/2}\left\Vert \nabla (u-\pi
_{p}u)\right\Vert _{0,\Omega },\left( \frac{\varepsilon _{1}}{\varepsilon
_{2}}\right) ^{1/2}\left\Vert \nabla (w-\pi _{p}w)\right\Vert _{0,\Omega
}\right\} \lesssim e^{-\tilde{\beta}p},
\end{equation*}%
provided Assumption \ref{assumption2D} holds.
\end{lemma}

\begin{proof} The proof is separated into two cases.

\textit{Case 1}: $\kappa p\frac{\varepsilon _{1}}{\varepsilon _{2}}\geq 1/2$
(asymptotic case).

In this case we use the asymptotic mesh $\Delta _{A}$ and $u$ satisfies (\ref%
{B1}). From \cite[Corollary 3.5]{MX}, we have $\pi _{p}u\in \mathcal{S}_{0}^{%
{p}}(\Delta _{A})$ such that 
\begin{equation}
\left\Vert u-\pi _{p}u\right\Vert _{\infty ,\Omega }+\left\Vert \nabla
(u-\pi _{p}u)\right\Vert _{\infty ,\Omega }\lesssim p^{2}(\ln
p+1)^{2}e^{-\beta p\kappa }\;,\;\tilde{\beta}\in \mathbb{R}^{+},
\label{aux1}
\end{equation}%
where we used the fact that for $u$ the boundary layers estimate (\ref{B1})
includes an `extra' power of $\frac{\varepsilon _{1}}{\varepsilon _{2}}$.
For $w=\varepsilon _{1}\Delta u$, we have 
\begin{equation*}
\left\Vert D^{\alpha }w\right\Vert _{0,\Omega }\lesssim \varepsilon
_{1}K^{|\alpha |+2}\max \{\left( |\alpha |+2\right) ^{|\alpha
|+2},\varepsilon _{1}^{1-\left( |\alpha |+2\right) }\}\;\;\forall \;|\alpha
|\in \mathbb{N}_{0}^{2}.
\end{equation*}%
and by \cite[Corollary 3.5]{MX}, there exists $\pi _{p}w\in \mathcal{S}^{{p}%
}(\Delta _{A})$ such that 
\begin{equation}
\left\Vert w-\pi _{p}w\right\Vert _{0,\Omega }+\left\Vert \nabla (w-\pi
_{p}w)\right\Vert _{0,\Omega }\lesssim p^{2}(\ln p+1)^{2}e^{-\beta p\kappa }.
\label{aux2}
\end{equation}%
This gives the result in the asymptotic case, once we absorb the powers of $%
p $ in the exponential term and adjust the constants. (Actucally, the proven
result is stronger than the Lemma's assertion.)

\textit{Case 2}: $\kappa p\frac{\varepsilon _{1}}{\varepsilon _{2}}<1/2$
(pre-asymptotic case).

In this case we use the \emph{Spectral Boundary Layer} mesh $\Delta
_{BL}(\kappa ,p)$ and $u$ is decomposed as in (\ref{B2}). Each component is
approximated separately, except for the remainders $R,\tilde{R}$, which are
already exponentially small (hence not approximated at all). The
approximations for the smooth parts $S,\tilde{S}$ are constructed as in Case
1 above (basically taken to be that of \cite{MS}) and estimates like (\ref%
{aux1}) may be obtained; the details are omitted. For the boundary layers,
we proceed similarly to Lemma 3.4 in \cite{MX}. So we only present the
arguments for the approximation of the layers. For $E_{1}$ we have that (\ref%
{B4}) holds, which gives%
\begin{equation*}
\left\vert \frac{\partial ^{n+m}E_{1}(\rho ,\theta )}{\partial \rho
^{n}\partial \theta ^{m}}\right\vert \lesssim K_{1}^{n}\frac{1}{\varepsilon
_{2}}\left( \frac{\varepsilon _{1}}{\varepsilon _{2}}\right)
^{1-n}e^{-2\kappa p}\;\;\forall \;(\rho ,\theta )\in \Omega \backslash
\left\{ \cup _{i=1}^{N_{1}}\Omega _{i}^{BL_{1}}\right\} .
\end{equation*}%
The above estimate, which shows that $E_{1}$ is exponentially small outside
the region $\cup _{i=1}^{N_{1}}\Omega _{i}^{BL_{1}}$, allows us to
approximate $E_{1}$ by its bilinear interpolant there. Inside this region,
we use the interpolant $\pi _{p}$ of \cite[Corollary 3.5]{MX}, to get%
\begin{equation*}
\left\Vert \nabla (E_{1}-\pi _{p}E_{1})\right\Vert _{0,\Omega
_{i}^{BL_{1}}}\lesssim K_{1}\frac{1}{\varepsilon _{2}}\kappa p\frac{%
\varepsilon _{1}}{\varepsilon _{2}}p^{2}(\ln p+1)^{2}e^{-\beta p\kappa
}\lesssim e^{-\tilde{\beta}p}.
\end{equation*}%
Similarly, for $E_{2}$ we have from (\ref{B4}),%
\begin{equation*}
\left\vert \frac{\partial ^{n+m}E_{2}(\rho ,\theta )}{\partial \rho
^{n}\partial \theta ^{m}}\right\vert \lesssim K_{2}^{n}\varepsilon
_{2}^{-n}e^{-2\kappa p}\;\;\forall \;(\rho ,\theta )\in \Omega \backslash
\left\{ \cup _{i=1}^{N_{1}}\Omega _{i}^{BL_{1}}\cup _{j=1}^{N_{1}}\Omega
_{j}^{BL_{2}}\right\} ,
\end{equation*}%
hence, $E_{2}$ will be approximated by its bilinear interpolant in $\Omega
\backslash \left\{ \cup _{i=1}^{N_{1}}\Omega _{i}^{BL_{1}}\cup
_{j=1}^{N_{1}}\Omega _{j}^{BL_{2}}\right\} $. It remains to approximate $%
E_{2}$ in $\cup _{i=1}^{N_{1}}\Omega _{i}^{BL_{1}}$ and $\cup
_{j=1}^{N_{1}}\Omega _{j}^{BL_{2}}.$ From \cite[Corollary 3.5]{MX},%
\begin{equation*}
\left\Vert \nabla (E_{2}-\pi _{p}E_{2})\right\Vert _{0,\Omega
_{i}^{BL_{1}}}\lesssim K_{1}\frac{1}{\varepsilon _{2}}\left( \kappa p\frac{%
\varepsilon _{1}}{\varepsilon _{2}}\right) ^{1/2}p^{2}(\ln p+1)^{2}e^{-\beta
p\kappa }\lesssim e^{-\tilde{\beta}p}
\end{equation*}%
and%
\begin{equation*}
\left\Vert \nabla (E_{2}-\pi _{p}E_{2})\right\Vert _{0,\Omega
_{j}^{BL_{2}}}\lesssim K_{1}\frac{1}{\varepsilon _{2}}\kappa p\left(
\varepsilon _{2}-\frac{\varepsilon _{1}}{\varepsilon _{2}}\right)
^{1/2}p^{2}(\ln p+1)^{2}e^{-\beta p\kappa }\lesssim \varepsilon
_{2}^{-1/2}e^{-\tilde{\beta}p}.
\end{equation*}%
For $\tilde{E}_{1}$ we have%
\begin{equation*}
\left\vert \frac{\partial ^{n+m}\tilde{E}_{1}(\rho ,\theta )}{\partial \rho
^{n}\partial \theta ^{m}}\right\vert \lesssim \tilde{K}_{1}^{n}\left( \frac{%
\varepsilon _{1}}{\varepsilon _{2}}\right) ^{-n}e^{-2\kappa p}\;\;\forall
\;(\rho ,\theta )\in \Omega \backslash \left\{ \cup _{i=1}^{N_{1}}\Omega
_{i}^{BL_{1}}\right\} ,
\end{equation*}%
and%
\begin{equation*}
\left\Vert \nabla (\tilde{E}_{1}-\pi _{p}\tilde{E}_{1})\right\Vert
_{0,\Omega _{i}^{BL_{1}}}\lesssim \left( \frac{\varepsilon _{1}}{\varepsilon
_{2}}\right) ^{-1}\left( \frac{\varepsilon _{1}}{\varepsilon _{2}}\right)
^{1/2}\kappa p\tilde{K}_{1}e^{-2\kappa p}\lesssim \left( \frac{\varepsilon
_{1}}{\varepsilon _{2}}\right) ^{-1/2}e^{-\tilde{\beta}p}
\end{equation*}%
For $\tilde{E}_{2}$ we have%
\begin{equation*}
\left\vert \frac{\partial ^{n+m}\tilde{E}_{2}(\rho ,\theta )}{\partial \rho
^{n}\partial \theta ^{m}}\right\vert \lesssim \tilde{K}_{1}^{n}\varepsilon
_{2}^{-n}e^{-\rho /\varepsilon _{2}}\;\;\forall \;(\rho ,\theta )\in \Omega
\backslash \left\{ \cup _{i=1}^{N_{1}}\Omega _{i}^{BL_{2}}\right\} ,
\end{equation*}%
and%
\begin{equation*}
\left\Vert \nabla (\tilde{E}_{2}-\pi _{p}\tilde{E}_{2})\right\Vert
_{0,\Omega _{i}^{BL_{1}}}\lesssim \tilde{K}_{1}\frac{1}{\varepsilon _{2}}%
\left( \kappa p\varepsilon _{2}\right) ^{1/2}p^{2}(\ln p+1)^{2}e^{-\beta
p\kappa }\lesssim \varepsilon _{2}^{-1/2}e^{-\tilde{\beta}p}.
\end{equation*}%
Combining the above gives the result, once we use the Sobolev embedding
theorem to handle the $L^{\infty }$ bounds.

\end{proof}

The previous lemma allows us to measure the error between the solution $%
(u,w) $ and its interpolant $(\pi _{p}u,\pi _{p}w)$. The following one
allows us to measure the error between the interpolant and the finite
element solution $(u_{N},w_{N})$.

\begin{lemma}
\label{lemma_tool}Assume \ref{assumption2D} holds and let $\left(
u_{N},w_{N}\right) \in \mathcal{S}_{0}^{{p}}(\Delta _{BL}(\kappa ,p))\times 
\mathcal{S}^{{p}}(\Delta _{BL}(\kappa ,p))$ be the solution to (\ref%
{discrete2}). Then there exist polynomials $\pi _{p}u\in \mathcal{S}_{0}^{{p}%
}(\Delta _{BL}(\kappa ,p))$, $\pi _{p}w\in \mathcal{S}^{{p}}(\Delta
_{BL}(\kappa ,p))$ such that 
\begin{equation*}
|||\left( \pi _{p}u-u_{N},\pi _{p}w-w_{N}\right) |||^{2}\lesssim e^{-\tilde{%
\beta p}},
\end{equation*}%
with $\tilde{\beta}>0$ a constant independent of $\varepsilon $ and $p$.
\end{lemma}

\begin{proof}

Recall that the bilinear form $\mathcal{A}\left( (\cdot ,\cdot ),(\cdot
,\cdot )\right) $, given by (\ref{Auw}) is coercive (see eq. (\ref{coercive}%
)), hence we have with $\psi =\pi _{p}u-u_{N}\in S_{0}^{p}(\Delta
_{BL}(\kappa ,p))$ and $\phi =\pi _{p}w-w_{N}\in \mathcal{S}^{p}(\Delta
_{BL}(\kappa ,p)),$ 
\begin{eqnarray*}
|||(\psi ,\phi )|||^{2} &\leq &\mathcal{A}\left( (\pi _{p}u-u,\pi
_{p}w-w),(\psi ,\phi )\right) =\left\langle c\left( \pi _{p}u-u\right) ,\psi
\right\rangle _{\Omega }+\varepsilon _{2}^{2}\left\langle \nabla \left( \pi
_{p}u-u\right) ,\nabla \psi \right\rangle _{\Omega } \\
&&+\left\langle \pi _{p}w-w,\phi \right\rangle _{\Omega }+\varepsilon
_{1}\left\langle \nabla \left( \pi _{p}u-u\right) ,\nabla \phi \right\rangle
_{\Omega }-\varepsilon _{1}\left\langle \nabla \left( \pi _{p}w-w\right)
,\nabla \psi \right\rangle _{\Omega } \\
&=:&I_{1}+I_{2}+I_{3}+I_{4}+I_{5}
\end{eqnarray*}%
Each term is treated using Cauchy-Schwarz and Lemma \ref{2Dapprox}, except
for $I_{4}$ which also requires the use of an inverse inequality:%
\begin{eqnarray*}
\left\vert I_{1}\right\vert  &=&\left\vert \left\langle c\left( \pi
_{p}u-u\right) ,\psi \right\rangle _{\Omega }\right\vert \lesssim \left\Vert
\pi _{p}u-u\right\Vert _{0,\Omega }\left\Vert \psi \right\Vert _{0,\Omega
}\lesssim e^{-\tilde{\beta}p}\left\Vert \psi \right\Vert _{0,\Omega }, \\
\left\vert I_{2}\right\vert  &=&\left\vert \varepsilon _{2}^{2}\left\langle
\nabla \left( \pi _{p}u-u\right) ,\nabla \psi \right\rangle _{\Omega
}\right\vert \lesssim \varepsilon _{2}^{2}\left\Vert \nabla \left( \pi
_{p}u-u\right) \right\Vert _{0,\Omega }\left\Vert \nabla \psi \right\Vert
_{0,\Omega } \\
&\lesssim &\varepsilon _{2}^{3/2}e^{-\tilde{\beta}p}\left\Vert \nabla \psi
\right\Vert _{0,\Omega } \\
\left\vert I_{3}\right\vert  &=&\left\vert \left\langle \pi _{p}w-w,\phi
\right\rangle _{\Omega }\right\vert \leq \left\Vert \pi _{p}u-u\right\Vert
_{0,\Omega }\left\Vert \phi \right\Vert _{0,\Omega }\lesssim e^{-\tilde{\beta%
}p}\left\Vert \phi \right\Vert _{0,\Omega } \\
\left\vert I_{4}\right\vert  &=&\left\vert \varepsilon _{1}\left\langle
\nabla \left( \pi _{p}u-u\right) ,\nabla \phi \right\rangle _{\Omega
}\right\vert \leq \varepsilon _{1}\left\Vert \nabla \left( \pi
_{p}u-u\right) \right\Vert _{0,\Omega }\left\Vert \nabla \phi \right\Vert
_{0,\Omega } \\
&\lesssim &\varepsilon _{1}\left\Vert \nabla \left( \pi _{p}u-u\right)
\right\Vert _{0,\Omega }(\kappa p\frac{\varepsilon _{1}}{\varepsilon _{2}}%
)^{-1}p^{2}\left\Vert \phi \right\Vert _{0,\Omega }\lesssim \varepsilon
_{1}\varepsilon _{2}^{-1/2}(\kappa p\frac{\varepsilon _{1}}{\varepsilon _{2}}%
)^{-1}e^{-\tilde{\beta}p}\left\Vert \phi \right\Vert _{0,\Omega } \\
&\lesssim &\varepsilon _{2}^{1/2}e^{-\tilde{\beta}p}\left\Vert \phi
\right\Vert _{0,\Omega } \\
\left\vert I_{5}\right\vert  &=&\left\vert \varepsilon _{1}\left\langle
\nabla \left( \pi _{p}w-w\right) ,\nabla \psi \right\rangle _{\Omega
}\right\vert \leq \varepsilon _{1}\left\Vert \nabla \left( \pi
_{p}w-w\right) \right\Vert _{0,\Omega }\left\Vert \nabla \psi \right\Vert
_{0,\Omega } \\
&\lesssim &\varepsilon _{1}\left( \frac{\varepsilon _{1}}{\varepsilon _{2}}%
\right) ^{-1/2}e^{-\tilde{\beta}p}\left\Vert \nabla \psi \right\Vert
_{0,\Omega } \\
&\lesssim &(\varepsilon _{1}\varepsilon _{2})^{1/2}e^{-\tilde{\beta}%
p}\left\Vert \nabla \psi \right\Vert _{0,\Omega }.
\end{eqnarray*}%
Hence, 
\begin{equation*}
|||(\psi ,\phi )|||^{2}\lesssim e^{-\tilde{\beta p}}\left( \left\Vert \psi
\right\Vert _{0,\Omega }+\left\Vert \nabla \psi \right\Vert _{0,\Omega
}+\left\Vert \phi \right\Vert _{0,\Omega }\right) \lesssim e^{-\tilde{\beta p%
}}|||(\psi ,\phi )|||
\end{equation*}%
and the proof is complete.

\end{proof}

We now present our main result.

\begin{theorem}
\label{thm_main} Let $(u,w)\in H_{0}^{1}(\Omega )\times H^{1}(\Omega ),$ $%
\left( u_{N},w_{N}\right) \in \mathcal{S}_{0}^{{p}}(\Delta _{BL}(\kappa
,p))\times \mathcal{S}^{{p}}(\Delta _{BL}(\kappa ,p))$ be the solutions to (%
\ref{mixed}) and (\ref{discrete2}) respectively, and suppose Assumption \ref%
{assumption2D} holds. Then there exists a positive constant $\tilde{\beta}$,
independent of $\varepsilon ,$ such that 
\begin{equation*}
|||\left( u-u_{N},w-w_{N}\right) |||\lesssim e^{-\tilde{\beta}p}.
\end{equation*}
\end{theorem}

\begin{proof} The triangle inequality gives 
\begin{equation*}
|||\left( u-u_{N},w-w_{N}\right) |||\leq |||\left( u-\pi _{p}u,w-\pi
_{p}w\right) |||+|||\left( \pi _{p}u-u_{N},\pi _{p}w-w_{N}\right) |||
\end{equation*}%
and we then use Lemmas \ref{2Dapprox} and \ref{lemma_tool}. \end{proof}

\subsection{A balanced norm}

As is well known (see, e.g., \cite{MX} and the references therein) the
energy norm is too weak and `does not see the layers', since for $%
u=S+E_{1}+E_{2}$, with $S$ the smooth part, $E_{1}$ the faster decaying
layer component and $E_{2}$ the slower decaying layer component, there holds%
\begin{equation*}
|||\left( S,\varepsilon _{1}\Delta S\right) |||^{2}\lesssim \varepsilon
_{1}^{2}+\varepsilon _{2}^{2}+1=O(1),
\end{equation*}%
\begin{equation*}
|||\left( E_{1},\varepsilon _{1}\Delta E_{1}\right) |||^{2}\lesssim 1+\frac{%
\varepsilon _{1}}{\varepsilon _{2}}+\frac{\varepsilon _{1}^{2}}{\varepsilon
_{2}^{3}}=O\left( \frac{\varepsilon _{1}}{\varepsilon _{2}}\right) ,
\end{equation*}%
\begin{equation*}
|||\left( E_{2},\varepsilon _{1}\Delta E_{2}\right) |||^{2}\lesssim \frac{%
\varepsilon _{1}^{2}}{\varepsilon _{2}^{3}}+\varepsilon _{2}=O\left(
\varepsilon _{2}\right) .
\end{equation*}%
Hence, as $\frac{\varepsilon _{1}}{\varepsilon _{2}},\varepsilon
_{2}\rightarrow 0$ the norm of the layer components tends to 0, which
manifests itself as `the method performing better as $\frac{\varepsilon _{1}%
}{\varepsilon _{2}},\varepsilon _{2}\rightarrow 0$' (see Section \ref{nr}).

The norm%
\begin{equation}
|||\left( u,w\right) |||_{B}^{2}:=\frac{\varepsilon _{2}}{\varepsilon _{1}}%
\left\Vert w\right\Vert _{0,\Omega }^{2}+\varepsilon _{2}\left\Vert \nabla
u\right\Vert _{0,\Omega }^{2}+\left\Vert u\right\Vert _{0,\Omega }^{2},
\label{balanced}
\end{equation}%
is balanced, since%
\begin{equation*}
|||\left( S,\varepsilon _{1}\Delta S\right) |||_{B}^{2}\lesssim \varepsilon
_{1}\varepsilon _{2}+\varepsilon _{2}+1=O(1),
\end{equation*}%
\begin{equation*}
|||\left( E_{1},\varepsilon _{1}\Delta E_{1}\right) |||_{B}^{2}\lesssim 1+%
\frac{\varepsilon _{1}}{\varepsilon _{2}^{2}}+\frac{\varepsilon _{1}^{2}}{%
\varepsilon _{2}^{3}}=O\left( 1\right) ,
\end{equation*}%
\begin{equation*}
|||\left( E_{2},\varepsilon _{1}\Delta E_{2}\right) |||_{B}^{2}\lesssim 
\frac{\varepsilon _{1}}{\varepsilon _{2}^{2}}+1+\varepsilon _{2}=O\left(
1\right) .
\end{equation*}%
The problem with (\ref{balanced}) is that the bilinear form is not coercive
with respect to this norm, and the proof of convergence in this stronger
norm remains open. Nevertheless, in Section \ref{nr} we show the results of
numerical computations using this norm as well.

\section{Numerical results}

\label{nr}

We consider the problem 
\begin{align*}
\varepsilon _{1}^{2}\Delta ^{2}u-\varepsilon _{2}^{2}\Delta u+u& =f\text{ in 
}\Omega , \\
u=\frac{\partial u}{\partial n}& =0\text{ on }\partial \Omega ,
\end{align*}%
where $f=10x$ and the domain $\Omega $ is the interior of the so called 
\emph{Cranioid}-curve, given by 
\begin{equation*}
\gamma (\theta )=\left( \frac{1}{4}\sin (\theta )+\frac{1}{2}\sqrt{1-0.9\cos
(\theta )^{2}}+\frac{1}{2}\sqrt{1-0.7\cos (\theta )^{2}}\right) \cdot 
\begin{pmatrix}
\cos (\theta ) \\ 
\sin (\theta )%
\end{pmatrix}%
\end{equation*}%
where $\theta \in \lbrack 0,2\pi )$, see also Figure~\ref{fig:cranioid},
where for rather large values of $\varepsilon _{1}$ and $\varepsilon _{2}$
the mesh is also shown. Parallel to the boundary the two mesh layers
corresponding to the solution decomposition are visible. We use $\kappa =1$
for all our computations.

\begin{figure}[h]
\begin{center}
\includegraphics[width=0.55\textwidth]{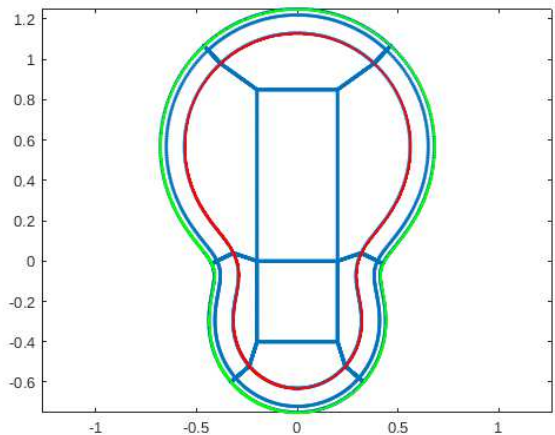}
\end{center}
\caption{The {\emph{Spectral Boundary layer}} mesh for the cranioid domain.}
\label{fig:cranioid}
\end{figure}

The exact solution to this problem is not known. Therefore, we use a
numerically computed reference solution as substitute, computed by
increasing the maximal polynomial degree by 2, adjusting the mesh and
recomputing the corresponding numerical solution. All computations were made
using the finite element library $\mathbb{SOFE}$ (%
\url{https://github.com/SOFE-Developers/SOFE}) running in Matlab/Octave. The
error in the energy norm $|||(u-u_{N},w-w_{N})|||$, will be plotted versus
the polynomial degree $p$, in a semi-log scale.

We look at two simulations: in the first we fix $\varepsilon _{1}=10^{-11}$
and vary $\varepsilon _{2}=10^{-j},j=3,4,5$ and show the results in Figure~%
\ref{fig:eg1} (left). We observe exponential convergence in the energy norm
as solid lines and even in the balanced norm (\ref{balanced}) as dashed
lines. Notice that for the energy norm, the error seems to be getting better
as $\varepsilon _{2}$ tends to $0$, which is a manifestation of the lack of
`balance', as discussed above. Next, we fix $\varepsilon _{2}=10^{-3}$ and
vary $\varepsilon _{1}=10^{-i},i=7,8,9,10$. In Figure~\ref{fig:eg1} (right)
the results are shown in the energy norm as solid lines and the balanced
norm as dashed lines. Again we observe exponential convergence, with all the
lines coinciding (for both norms). This is in agreement with \cite{FR3}
where the balanced norm (\ref{balanced}) includes this extra power of $%
\varepsilon _{2}$. 
\begin{figure}[h]
\begin{center}
\includegraphics[width=0.45\textwidth]{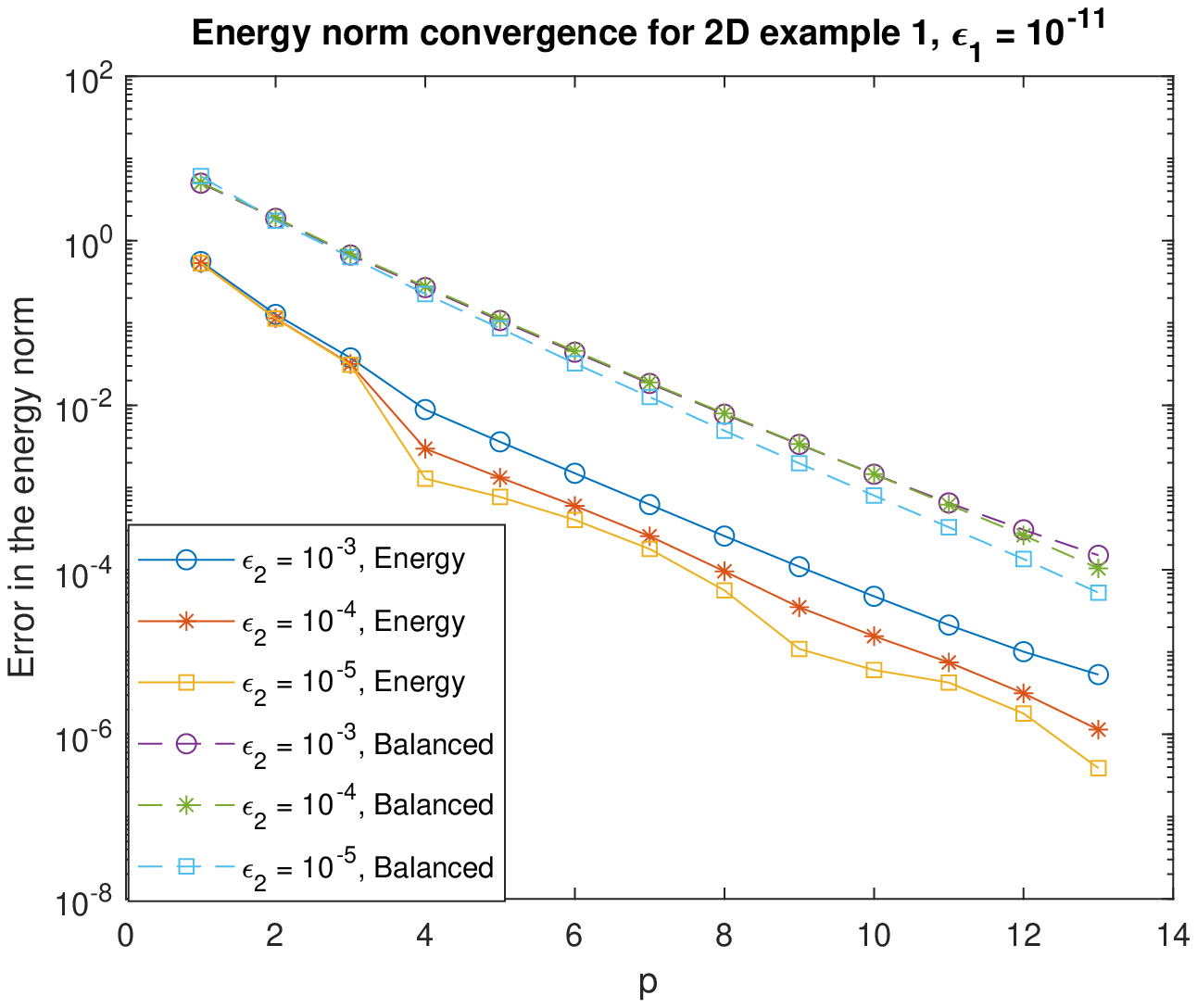} \mbox{} %
\includegraphics[width=0.45\textwidth]{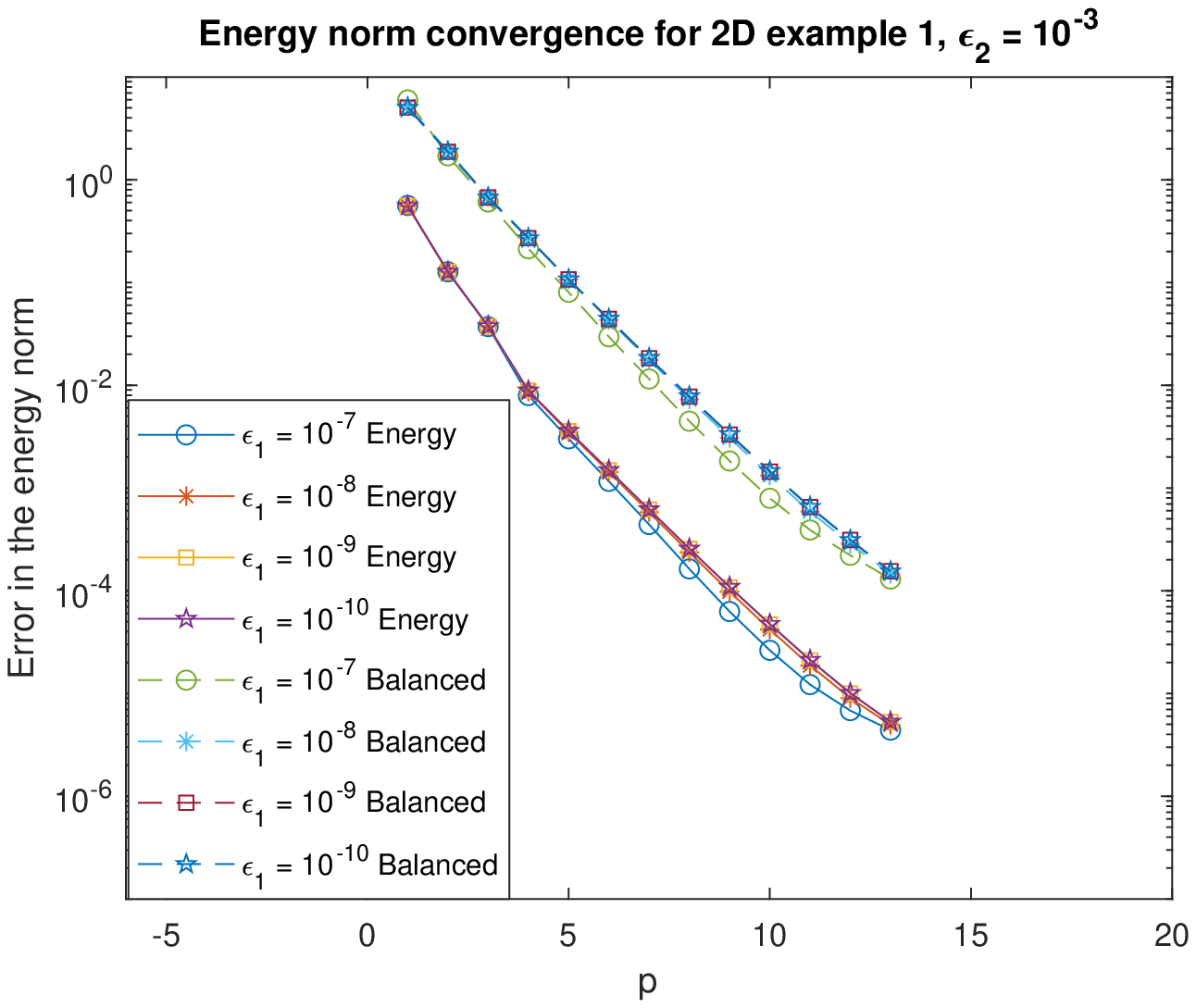}
\end{center}
\caption{Example 1 with $\protect\varepsilon _{1}$ fixed (left), and $%
\protect\varepsilon _{2}$ fixed (right).}
\label{fig:eg1}
\end{figure}
Thus our simulations reflect quite nicely the theoretical findings.
Furthermore, they hint at a stronger convergence result in the balanced
norm, which to prove is an open question.

Finally, as a second example we consider the case when $f$ is given by 
\begin{equation*}
f(x,y)=\frac{1}{\sqrt{(x+1/2)^{2}+y^{2}}}
\end{equation*}
which has a singularity just outside of $\Omega $. This causes the solution
to be less regular and any `lack of balance' phenomena should not be
visible. Figure \ref{fig:eg2} shows the results of this computation. On the
left, $\varepsilon _{1}$ is fixed and on the right $\varepsilon _{2}$ is
fixed, while in both graphs solid lines depict the error in the energy norm
and dashed lines the balanced norm. We observe that in all cases the
exponential convergence is robust with respect to both parameters. 
\begin{figure}[h]
\begin{center}
\includegraphics[width=0.45\textwidth]{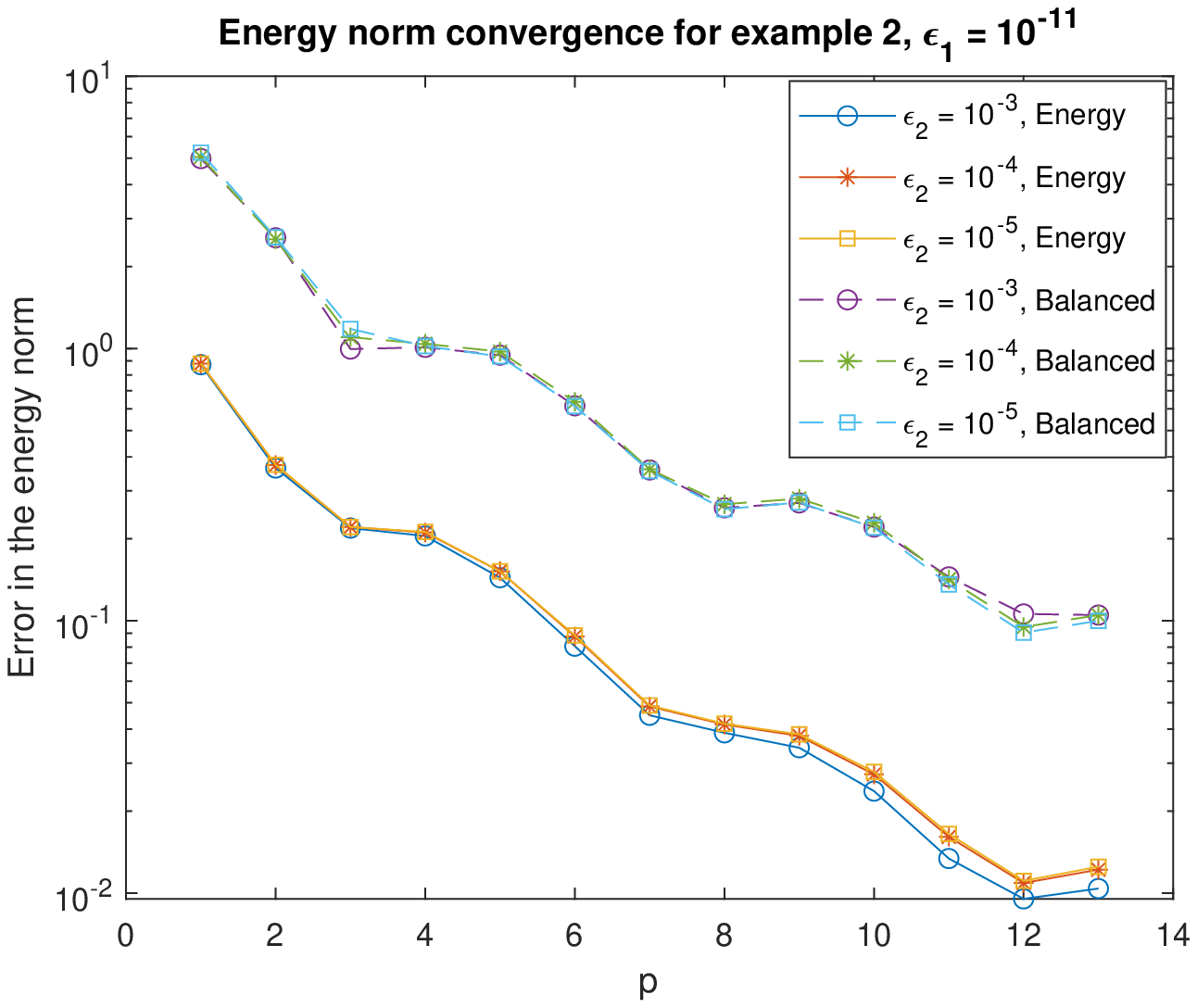} \mbox{} %
\includegraphics[width=0.45\textwidth]{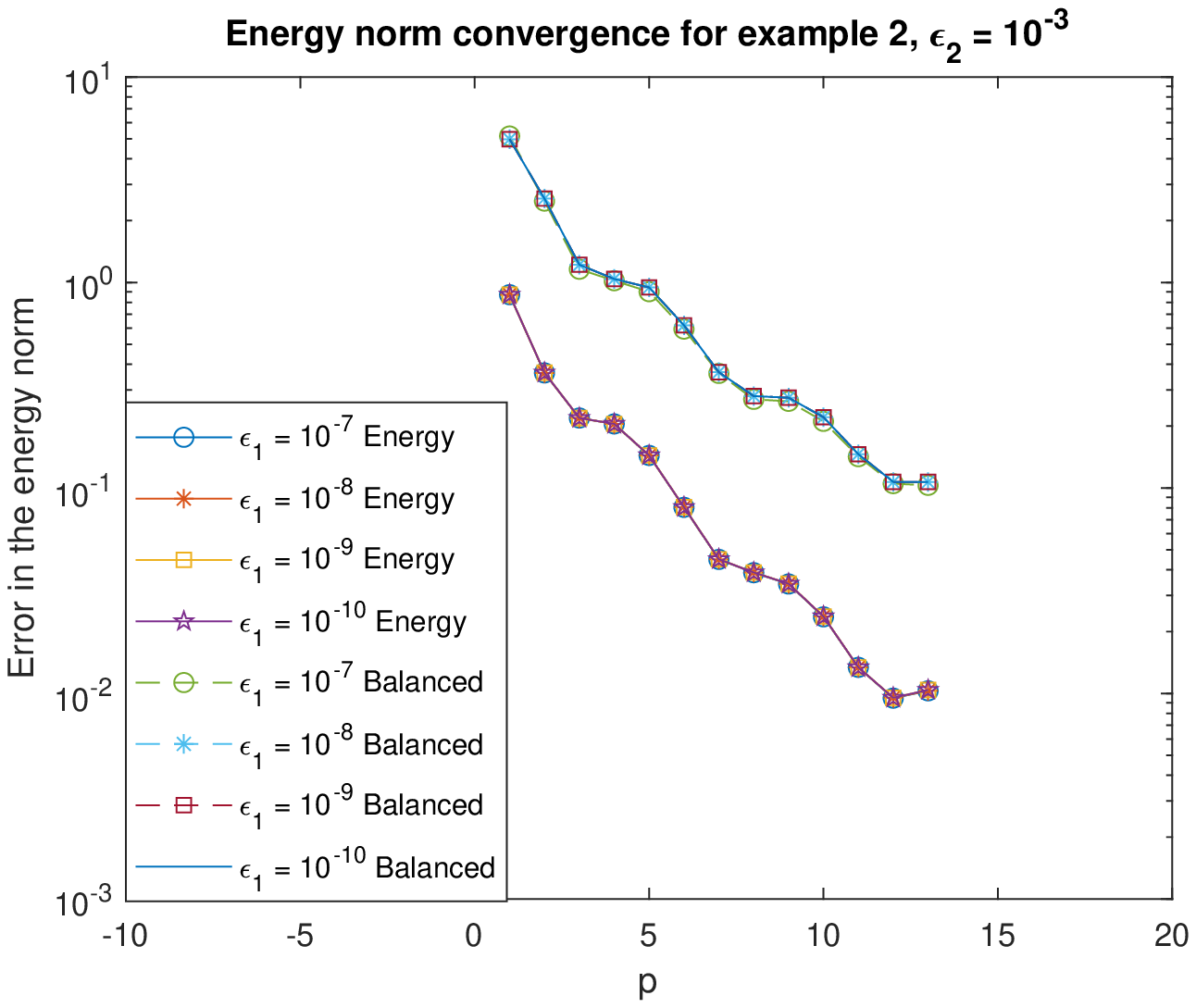}
\end{center}
\caption{Example 2 with $\protect\varepsilon _{1}$ fixed (left), and $%
\protect\varepsilon _{2}$ fixed (right).}
\label{fig:eg2}
\end{figure}

\end{document}